\documentclass{article}[11pt]
\usepackage{amsthm,amsmath}
\usepackage{amssymb}
\input xypic
\usepackage[all,tips]{xy}
\newtheorem{theo}[subsection]{Theorem}
\newtheorem{lem}[subsection]{Lemma}
\newtheorem{prop}[subsection]{Proposition}
\newtheorem{corr}[subsection]{Corollary}

\newtheorem{rem}[subsection]{Remark}

\newcommand{\cC}{{\cal C}}
\newcommand{\cZ}{{\cal Z}}
\newcommand{\caD}{{\cal D}}

\newcommand{\cM}{{\cal M}}
\newcommand{\cN}{{\cal N}}
\newcommand{\Rep}{{\cal R}{\it ep}}
\newcommand{\Vect}{{\cal V}{\it ect}}
\newcommand{\End}{{\cal E}{\it nd}}

\newcommand{\sve}{{\scriptscriptstyle {\vee}}}

\begin{document}
\author{Alexei Davydov}
\title{Centre of an algebra}
\maketitle
\date{}
\begin{center}
Department of Mathematics, Division of Information and Communication Sciences, Macquarie University, Sydney, NSW 2109, Australia
\end{center}
\begin{center}
davydov@science.mq.edu.au
\end{center}
\begin{abstract}
Motivated by algebraic structures appearing in Rational Conformal Field Theory we study a construction associating to an algebra in a monoidal category a commutative algebra ({\em full centre}) in the monoidal centre of the monoidal category. We establish Morita invariance of this construction by extending it to module categories. 

As an example we treat the case of group-theoretical categories.
\end{abstract}
\tableofcontents
\section{Introduction}
The notion of vector space with an associative product, i.e. an associative algebra, plays at important role in many parts of mathematics. 
Centre of algebra is an important invariant. For example, it tells when algebras can be Morita equivalent: if two algebras are Morita equivalent their centers are isomorphic (this follows from the fact that the centre of an algebra can be derived from the category of its modules). The notion of an algebra (and its module) can be transported to the much more general environment of monoidal categories. There it also plays an important role, capturing very diverse constructions (e.g. the notion of a monad is just a reincarnation of algebra). Although very straightforward with algebras and modules, the transportation of notions to the world of monoidal categories becomes less trivial with the centre. Without assuming commutativity of the tensor product it becomes hard to even define what is for an algebra to be commutative. Even with a commutativity assumption the situation is quite interesting, e.g. in a braided monoidal category there are two notions (left and right) of centre of an algebra. 

Our motivation for studying (and even defining) centers of algebras comes from Rational Conformal Field Theories (RCFTs). It was known for quite a while that a lot of information about the chiral half of a RCFT is contained in a certain monoidal category. Axiomatised in \cite{ms,tu} under the name of modular category, they were studied extensively by mathematicians and theoretical physicists. 
Recently it was realised that certain algebras (more precisely their categories of modules) in the chiral modular category of an RCFT correspond to a consistent set of its boundary conditions, while certain commutative algebras in the monoidal centre of the chiral modular category describe the RCFT in the bulk, i.e. the full RCFT (see \cite{kr} and references therein). The transition from algebras in a modular category to commutative algebras in its monoidal centre was studied in \cite{ffrs,kr} under the name of full centre. Although working very well (e.g. being Morita invariant) the construction uses heavily specific properties of algebras and modular categories.

In this paper we present a construction (also named full centre), which associates to an algebra in a monoidal category a commutative algebra ({\em full centre}) in the monoidal centre of the monoidal category (section \ref{fc}). Based on a universal property, the construction is quite general. We prove that full centre is Morita invariant by extending the definition from algebras to module categories over a monoidal category (section \ref{fc}). We also show that, when applied to algebras in a modular category, our construction give the right answer (section \ref{fc}). We conclude by looking at a case where the category is not modular, i.e. we describe full centers of separable algebras in categories of group-graded vector spaces and categories of representations of a group (section \ref{gth}).

For the definitions of monoidal categories and monoidal functors see \cite{mc}. 
Throughout the paper we assume that all monoidal categories are strict. This assumption is in fact inessential, it is made for simplicity and can be lifted without a problem. The term ``monoidal functor" means strong (otherwise we use ``lax monoidal functor"). We also assume that all monoidal functors are strong. We will often omit (especially in big diagrams) the tensor product sign, e.g. $XY$ will mean $X\otimes Y$, $X^2$ will mean $X^{\otimes 2}$ and $fg$ will mean $f\otimes g$ for objects $X,Y$ and morphisms $f,g$ of a monoidal category.

\section*{Acknowledgment}

The work on the paper began during author's visit to Max Planck Institut F\"ur Mathematik (Bonn) in summer 2008, was continued in 2009 in S\~ao Paulo, where the author was accommodated by the Institute of Mathematics and Statistics of the University of S\~ao Paulo and finished, while the author was visiting Macquarie University. The author would like to thank these institutions for hospitality and excellent working conditions. The author would like to thank ARC, Max Planck Gesellschaft, FAPESP (grant no. 2008/10526-1), and Safety Net Grant of Macquarie University, whose  financial support made the visits to Bonn, S\~ao Paulo and Sydney possible. During the work on the paper the author was consulting with J. Fuchs, M. M\"uger, C. Schweigert to whom he expresses his gratitude. Special thanks are to L. Kong, I. Runkel, R. Street and the anonymous referee. 

\section{Algebras in monoidal categories}

An (associative, unital) {\em algebra} in a (strict) monoidal category $\cC$ is a triple $(A,\mu,\iota)$ consisting of an object
$A\in\cC$ together with a {\em multiplication} $\mu:A\otimes A\to A$ and a {\em unit} map $\iota:I\to A$, satisfying
{\em associativity} and {\em unit} axioms:
$$
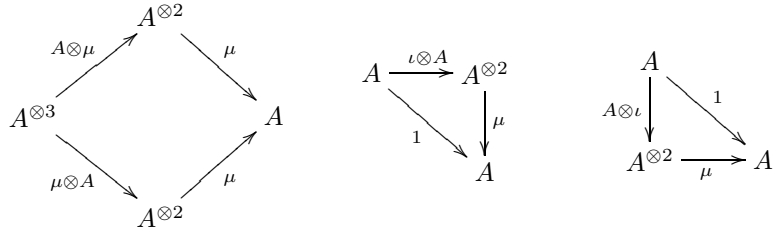
\begin{xy}
(-18,0)*+!R{\xybox{ \xymatrix{  & A^{\otimes 2}
\ar[dr]^{\mu} & \\ A^{\otimes 3} \ar[ru]^{A\otimes\mu}  \ar[rd]_{\mu\otimes A} & & A
\\  & A^{\otimes 2} \ar[ru]_{\mu} & }}}, (0,0)*+{\xybox{ \xymatrix{ A \ar[r]^{\iota\otimes A} \ar[dr]_1& A^{\otimes 2} \ar[d]^\mu
\\ & A }}}, (20,0)*+!L {\xybox{\xymatrix{ A \ar[d]_{A\otimes\iota} \ar[dr]^1& \\ A^{\otimes 2} \ar[r]_\mu
& A }}}
\end{xy}
$$
Where it will not cause confusion we will be talking about an algebra $A$, suppressing its multiplication and unit
maps.
\newline
A morphism of algebras $f:A\to B$ is a (unital) {\em homomorphism} if the following diagrams commute
$$
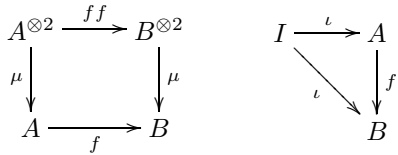
\begin{xy}
(-18,0)*+!R{\xybox{ \xymatrix{ A^{\otimes 2} \ar[r]^{ff} \ar[d]_\mu & B^{\otimes 2} \ar[d]^\mu \\ A \ar[r]_f & B }}}, (0,0)*+{\xybox{ \xymatrix{ I \ar[r]^{\iota} \ar[dr]_\iota & A \ar[d]^f
\\ & B }}}
\end{xy}
$$

An algebra $C$ in a braided monoidal category $\caD$ is {\em commutative} if the diagram 
$$\xymatrix{C\otimes C \ar[rr]^{c_{C,C}} \ar[rd]_\mu && C\otimes C \ar[ld]^\mu \\ & C  }$$

The unit object $I$ of a monoidal category has a canonical structure of an algebra. Here we another example of algebras, which will be used extensively. Recall (say from \cite{mc}) that an object $T$ of a category $\cC$ is {\em terminal} if for any object $X\in\cC$ there is exactly one morphism $X\to T$.
\begin{lem}\label{ta}
The terminal object of a monoidal category is an algebra. 
\newline
The terminal object of a braided monoidal category is a commutative algebra.
\end{lem}
\begin{proof}
Let $T$ be the terminal object. The unique morphisms $I\to T$, $T\otimes T\to T$ turn it into an algebra. Indeed, the axioms follow from the uniqueness  of morphisms into $T$. 
\end{proof}
Let $F:\cC\to\caD$ be a functor and $A\in\caD$ be an object. {\em Comma category} $F$$\downarrow$$A$ (see \cite{mc}) is the category of pairs $(X,x)$, where $X$ is an object of $\cC$ and $x:F(X)\to A$ is a morphism in $\caD$. Morphisms of pairs are morphisms of the first components, compatible with the second components. Note that if the functor $F$ is monoidal (assuming that $\cC,\caD$ are monoidal) and $A$ is an algebra then the comma-category $F$$\downarrow$$A$ is monoidal with the tensor product $(X,x)\otimes (Y,y) = (X\otimes Y,\overline{x\otimes y})$, where $\overline{x\otimes y}$ is the composition 
$$\xymatrix{F(X\otimes Y) \ar[r]^(.45){F_{X,Y}} & F(X)\otimes F(Y) \ar[r]^(.6){xy} & A\otimes A \ar[r]^(.6)\mu & A  }.$$ 
The unit object is $(I,i)$, where $i$ is the composition
$$\xymatrix{F(I) \ar[r] & I \ar[r]^\iota & A}.$$
The forgetful functor $F$$\downarrow$$A\to\cC$ and the evaluation functor $F$$\downarrow$$A\to\caD$ are monoidal. The following statement will also be used throughout.
\begin{lem}\label{cc}
Let $F:\cC\to\caD$ be a monoidal functor and $A$ be an algebra in $\caD$. Let $(B,b)$ be an algebra in the comma category $F$$\downarrow$$A$. Then $b:F(B)\to A$ is a homomorphism of algebras in $\caD$.
\end{lem}
\begin{proof}
Follows from the definition of morphisms and tensor product in comma category:
$$\xymatrix{F(B)\otimes F(B) \ar[rr]^{bb}  \ar@/_15pt/[ddr]_{\mu_{F(B)}} && A\otimes A \ar[d]^{\mu_A} \\ & F(B\otimes B) \ar[ul]^{F_{B,B}} \ar[r]^{\overline{bb}} \ar[d]_{F(\mu_B)}  & A \\ & F(B)\ar[ur]_b}$$
\end{proof}

\section{Monoidal centre of a category}
Here we recall (from \cite{js}) the construction and basic properties of monoidal centre of a monoidal category.

The {\em monoidal centre} $\cZ(\cC)$ of a monoidal category $\cC$ is the category of pairs
$(Z,z)$, where $Z\in\cC$ and $z$ stands for a natural collection of isomorphisms $z_X:Z\otimes X\to X\otimes Z$ ({\em half braidings}),
such that $z_I=1$ and the diagram
$$\xymatrix{ & Z\otimes(X\otimes Y) \ar[rr]^{z_{X\otimes Y}} \ar[ld]_{a_{Z,X,Y}} && (X\otimes Y)\otimes Z & \\ (Z\otimes
X)\otimes Y \ar[rd]_{z_X\otimes Y} &&&& X\otimes(Y\otimes Z) \ar[ul]^{a_{X,Y,Z}} \\ & (X\otimes Z)\otimes Y
\ar[rr]_{a_{X,Z,Y}^{-1}} && X\otimes(Z\otimes Y) \ar[ru]_{X\otimes z_Y} }$$ commutes for all $X,Y\in \cC$. Morphisms
in $\cZ(\cC)$ are morphisms of first components (in $\cC$), compatible, in a natural way, with second components. The
category $\cZ(\cC)$ is monoidal with respect to the tensor product
$$(Z,z)\otimes(W,w) = (Z\otimes W,z|w),$$
where $z|w$ is defined by
$$\xymatrix{ & (Z\otimes W)\otimes X \ar[rr]^{(z|w)_X} && X\otimes (Z\otimes W)  \ar[rd]^{a_{X,Z,W}} & \\
Z\otimes(W\otimes X) \ar[ru]^{a_{Z,W,X}}  \ar[rd]_{Z\otimes w_X} &&&& (X\otimes Z)\otimes W  \\ & Z\otimes (X\otimes W)
\ar[rr]_{a_{Z,X,W}} && (Z\otimes X)\otimes W \ar[ru]_{z_X\otimes W} }$$ Moreover $\cZ(\cC)$ is a braided monoidal
category with the braiding
$$c_{(X,x),(Y,y)} = x_Y.$$
\newline
The forgetful functor $$F:\cZ(\cC)\to\cC,\quad (Z,z)\mapsto Z$$ is clearly faithful and monoidal (with the monoidal structure being the identity).
In what follows we, when speaking about objects of the monoidal centre, will often omit the half braiding, e.g. instead of $(Z,z)$ we will have $Z$ (suppressing the half braiding $z$). 

\section{Full centre of an algebra}\label{fc}

Let $A$ be an algebra in a monoidal category $\cC$. The {\em full centre} $Z(A)$ of $A$ is an object of the monoidal centre $\cZ(\cC)$ together with a morphism $Z(A)\to A$ in $\cC$, terminal among pairs $(Z,\zeta)$, where $Z\in\cZ(\cC)$ and $\zeta:Z\to A$ is a morphism in $\cC$ such that the following diagram commutes:
\begin{equation}\label{cp}
\xymatrix{ Z\otimes A \ar[r]^{\zeta A} \ar[dd]_{z_A} & A\otimes A \ar[rd]^\mu\\ & & A\\ A\otimes Z \ar[r]_{A\zeta} & A\otimes A \ar[ur]_\mu }
\end{equation}
Here $z_A$ is the half-braiding of $Z$ as an object of $\cZ(\cC)$.  The terminality condition means that for any such pair $(Z,\zeta)$ there is a unique morphism $Z\to Z(A)$ in the monoidal centre $\cZ(\cC)$, which makes the diagram
$$\xymatrix{ Z \ar[rr] \ar[rd]_\zeta && Z(A) \ar[ld] \\ & A & }$$ commute. 

\begin{prop}\label{algz}
The full centre $Z(A)$ has a unique structure of an algebra in $\cZ(\cC)$ such that the morphism $Z(A)\to A$ is a homomorphism of algebras in $\cC$. Moreover $Z(A)$ is a commutative algebra in $\cZ(\cC)$.
\end{prop}
\begin{proof}
The pair $(I,\iota)$, where $\iota:I\to A$ is the unit map, satisfies to the condition (\ref{cp}) because the diagram 
$$\xymatrix{ I\otimes A \ar[rr]^{\iota A} \ar[dd] \ar[dr] && A\otimes A \ar[rd]^\mu\\ & A\ar[rr]^1&& A\\ A\otimes I \ar[rr]_{A\iota} \ar[ru] && A\otimes A \ar[ur]_\mu }$$ commutes. Hence there is a unique morphism $I\to Z(A)$ such that the diagram 
$$\xymatrix{ I \ar[rr] \ar[rd]_\iota && Z(A) \ar[ld]^\zeta \\ & A & }$$ commutes. 

Similarly, the morphism 
$$\xymatrix{Z(A)^{\otimes 2} \ar[r]^{\zeta\zeta} & A^{\otimes 2} \ar[r]^\mu & A}$$ satisfies condition (\ref{cp}). Indeed the diagram 
$$
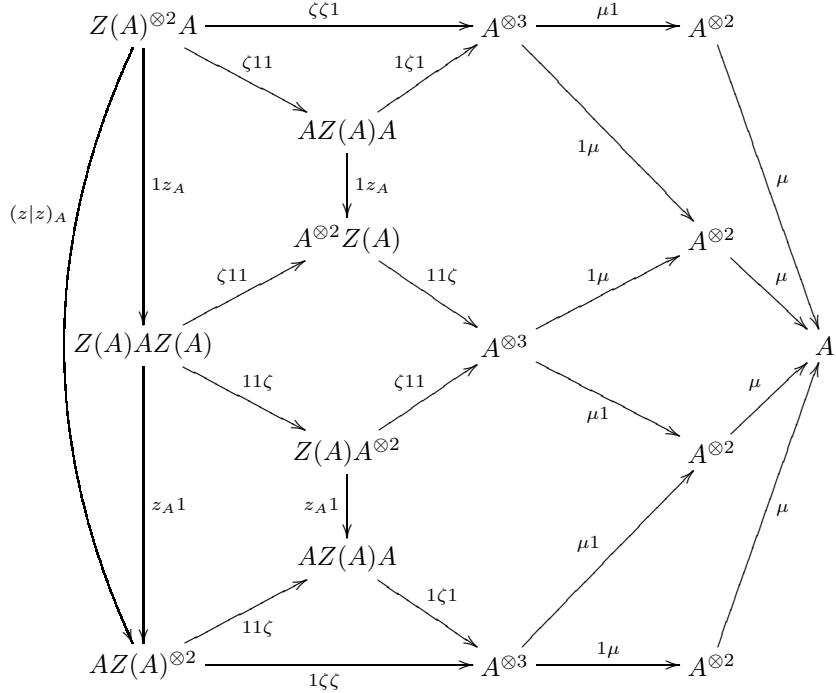
\xymatrix{
Z(A)^{\otimes 2}A \ar[rr]^{\zeta\zeta 1} \ar@/_30pt/[dddddd] _(.3){(z|z)_A} \ar[ddd]^{1z_A} \ar[rd]^{\zeta 11} && A^{\otimes 3} \ar[rr]^{\mu 1} \ar[rrdd]_{1\mu} && A^{\otimes 2} \ar[rddd]^\mu \\ 
& AZ(A)A \ar[ru]^{1\zeta 1} \ar[d]^{1z_A}  \\
& A^{\otimes 2}Z(A) \ar[dr]^{11\zeta} &&& A^{\otimes 2} \ar[rd]^\mu \\ 
Z(A)AZ(A) \ar[ddd]^{z_A1} \ar[ru]^{\zeta 11} \ar[rd]^{11\zeta} && A^{\otimes 3} \ar[drr]_{\mu 1} \ar[rru]^{1\mu} && & A\\ 
& Z(A)A^{\otimes 2} \ar[ur]^{\zeta 11} \ar[d]_{z_A1} &&& A^{\otimes 2} \ar[ru]^\mu \\ 
& AZ(A)A \ar[rd]^{1\zeta 1}  \\ 
AZ(A)^{\otimes 2} \ar[rr]_{1\zeta\zeta} \ar[ru]_{11\zeta} && A^{\otimes 3} \ar[rr]^{1\mu} \ar[rruu]^{\mu 1} && A^{\otimes 2} \ar[uuur]_\mu 
}$$ commutes. Thus, by the universal property,  there is a unique morphism $Z(A)^{\otimes 2}\to Z(A)$ such that the diagram 
$$\xymatrix{ Z(A)^{\otimes 2} \ar[r] \ar[d]_{\zeta\zeta} & Z(A) \ar[d]^\zeta \\ A^{\otimes 2} \ar[r]_\mu &  A}$$ commutes. 

Associativity, commutativity and unit axioms for $Z(A)$ follow from the uniqueness property. To prove associativity all we need to do it to show that the compositions 
$$\xymatrix{ Z(A)^{\otimes 3} \ar[r]^{\mu 1} & Z(A)^{\otimes 2} \ar[r]^\mu & Z(A) }$$
$$\xymatrix{ Z(A)^{\otimes 3} \ar[r]^{\mu 1} & Z(A)^{\otimes 2} \ar[r]^\mu & Z(A) }$$
coincide after being composed with $\zeta$. This is guaranteed by the commutative diagram:
$$
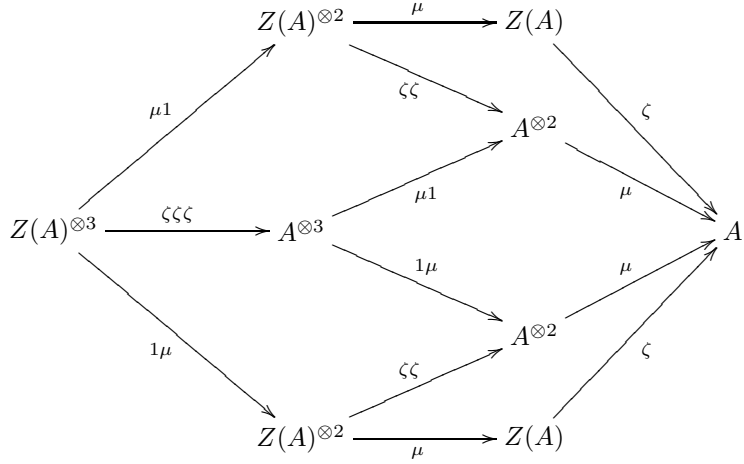
\xymatrix{ 
&& Z(A)^{\otimes 2} \ar[rr]^\mu \ar[rrd]_{\zeta\zeta} && Z(A) \ar[rrdd]^\zeta && \\
&&&& A^{\otimes 2} \ar[rrd]_\mu && \\
Z(A)^{\otimes 3} \ar[rruu]^{\mu 1} \ar[rr]^{\zeta\zeta\zeta} \ar[rrdd]_{1\mu} && A^{\otimes 3} \ar[rru]_{\mu 1} \ar[rrd]^{1\mu} &&&& A\\
&&&& A^{\otimes 2} \ar[rru]^\mu && \\
&& Z(A)^{\otimes 2} \ar[rru]^{\zeta\zeta} \ar[rr]_\mu && Z(A) \ar[rruu]_\zeta 
}$$
Similarly, to prove commutativity we need to show that $\mu z_{Z(A)}:Z(A)^{\otimes 2}\to Z(A)$ coincides with $\mu:Z(A)^{\otimes 2}\to Z(A)$ after being composed with $\zeta$. This follows from commutativity of the diagram:
$$
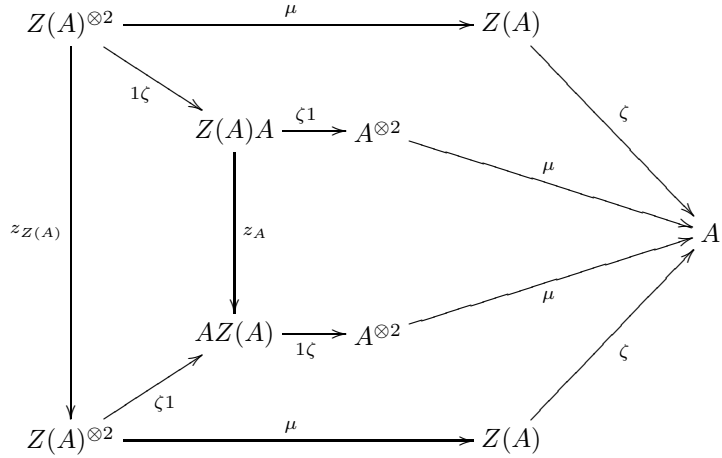
\xymatrix{
Z(A)^{\otimes 2} \ar[rrr]^\mu \ar[rd] _{1\zeta} \ar[dddd]_{z_{Z(A)}} &&& Z(A) \ar[rrdd]^\zeta \\
& Z(A)A \ar[r]^{\zeta 1} \ar[dd]^{z_A} & A^{\otimes 2} \ar[rrrd]^\mu \\
&&&&& A\\
& AZ(A) \ar[r]_{1\zeta}  & A^{\otimes 2} \ar[rrru]_\mu \\
Z(A)^{\otimes 2} \ar[rrr]^\mu \ar[ru] _{\zeta 1}  &&& Z(A) \ar[rruu]_\zeta  }$$
Finally, for (one of) the unit axioms it is enough to check that, after composing with $\zeta$, $\mu(\iota 1)$ coincides with $\zeta$. This is guaranteed by the commutative diagram:
$$\xymatrix{ 
Z(A) \ar[rrrr]^\zeta \ar[drr]_\zeta \ar[ddr]^{\iota 1} \ar[dddd]_{\iota 1} &&&& A \\
&& A \ar[dd]_{\iota 1} \ar[urr]_1 \\
& AZ(A) \ar[rd]^{1\zeta} \\
&& A^{\otimes 2} \ar[rruuu]_\mu \\
Z(A)^{\otimes 2} \ar[ruu]^{\zeta 1} \ar[urr]_{\zeta\zeta} \ar[rrrr]_\mu &&&& Z(A) \ar[uuuu]_\zeta
}$$
\end{proof}
\begin{rem}\label{subc}
\end{rem}
Note that the category $\cZ(A)$ of pairs $(Z,\zeta)$, where $Z$ belongs to $\cZ(\cC)$ and $\zeta:Z\to A$ satisfies condition (\ref{cp}), is a full monoidal subcategory of the comma category $F$$\downarrow$$A$ for the forgetful functor $F:\cZ(\cC)\to\cC$. Indeed, tensor product $(ZT,\zeta\tau)$ of two such pairs $(Z,\zeta)$ and $(T,\tau)$ again has this property:
$$\xymatrix{
ZTA \ar[rr]^{\zeta\tau 1} \ar@/_30pt/[dddddd] _(.3){(z|t)_A} \ar[ddd]^{1t_A} \ar[rd]^{\zeta 11} && A^{\otimes 3} \ar[rr]^{\mu 1} \ar[rrdd]_{1\mu} && A^{\otimes 2} \ar[rddd]^\mu \\ 
& ATA \ar[ru]^{1\tau 1} \ar[d]^{1t_A}  \\
& A^{\otimes 2}T \ar[dr]^{11\tau} &&& A^{\otimes 2} \ar[rd]^\mu \\ 
ZAT \ar[ddd]^{z_A1} \ar[ru]^{\zeta 11} \ar[rd]^{11\tau} && A^{\otimes 3} \ar[drr]_{\mu 1} \ar[rru]^{1\mu} && & A\\ 
& ZA^{\otimes 2} \ar[ur]^{\zeta 11} \ar[d]_{z_A1} &&& A^{\otimes 2} \ar[ru]^\mu \\ 
& AZA \ar[rd]^{1\zeta 1}  \\ 
AZT \ar[rr]_{1\zeta\tau} \ar[ru]_{11\tau} && A^{\otimes 3} \ar[rr]^{1\mu} \ar[rruu]^{\mu 1} && A^{\otimes 2} \ar[uuur]_\mu 
}$$
Now the major part of proposition \ref{algz} follows from lemmas \ref{ta},\ref{cc}.

\section{Left centres}\label{}

Now let $\caD$ be a braided monoidal category. Following \cite{os,vz} we define the {\em left centre} $C_l(B)$ of an algebra $B$ in $\caD$ as the terminal object in the category of morphisms $y:Y\to B$ such that the following diagram commutes: 
\begin{equation}\label{lcp}
\xymatrix{ Y\otimes B \ar[r]^{yB} \ar[dd]_{c_{Y,B}} & B\otimes B \ar[rd]^\mu\\ & & B\\ B\otimes Y \ar[r]_{By} & B\otimes B \ar[ur]_\mu }
\end{equation}
Similarly, one can define the {\em right centre} of an algebra in a braided monoidal category.
\begin{prop}\label{algl}
The left centre $C_l(B)$ has a unique structure of algebra in $\caD$ such that the morphism $C_l(B)\to B$ is a homomorphism of algebras in $\caD$. Moreover $C_l(B)$ is a commutative algebra in $\caD$. 
\newline
Similarly for the right centre.
\end{prop}
\begin{proof}
Analogous to the proof of .
\end{proof}
\begin{rem}
\end{rem}
The major part of proposition \ref{algl} follows from lemmas \ref{ta},\ref{cc} if we note that the category $\cC_l(B)$ of pairs $(Y,y)$, with $y:Y\to B$ satisfying condition (\ref{lcp}), is a full monoidal subcategory of the comma category $id_{\cZ(\caD)}$$\downarrow$$B$ for the identity functor $id:\cZ(\caD)\to\cZ(\caD)$.
\medskip

Now assume that the forgetful functor $F:\cZ(\cC)\to\cC$ has a right adjoint $R:\cC\to\cZ(\cC)$ with the natural transformations of the adjunction:
$$\alpha_U:U\to RF(U),\quad \beta_X:FR(X)\to X,\quad U\in\cZ(\cC), X\in\cC.$$
Note that $R$ is automatically lax monoidal, i.e. it is equipped with the morphism $I\to R(I)$ and the natural transformation $R_{X,Y}:R(X)\otimes R(Y)\to R(X\otimes Y)$, which satisfy usual coherence axioms of a monoidal functor, but are not necessarily isomorphisms. 
\newline
Indeed, the morphism is given by the composite
$$\xymatrix{I \ar[r]^(.4){\alpha_I} & RF(I) \ar[r] & R(I) }$$ while the natural transformation is 
$$\xymatrix{ R(X)\otimes R(Y) \ar[d]_{\alpha_{R(X)R(Y)} } && R(X\otimes Y)\\ RF(R(X)\otimes R(Y)) \ar[rr]^{R(F_{R(X),R(Y)})} && R(FR(X)\otimes FR(Y)) \ar[u]_{R(\beta_X\beta_Y)}  }$$ Here $F_{U,V}:F(U\otimes V)\to F(U)\otimes F(V)$ is the natural isomorphism (the monoidal structure of $F$), which is in fact the identity for the forgetful functor $F$. 

The following statement is well-known. We add a proof for the sake of completeness.
\begin{lem}\label{monadj}
The adjunction natural transformations $\alpha$ and $\beta$ are monoidal.
\end{lem}
\begin{proof}
Monoidality of $\beta$ follows from the commutative diagram:
$$\xygraph{
*+{XY}="l" 
(
:@/^30pt/[u(2.5)r(6)] *+{RF(XY)} ^{\alpha_{XY}}
 (
 :@/^25pt/[d(2.5)r(2.5)] *+{R(F(X)F(Y))}="r" ^{R(F_{X,Y})}
 :@{=}[l(2.5)u(1)] *+{R(F(X)F(Y))}
 :[d(2)] *+{R(FRF(X)FRF(Y))}="d" _(.3){R(F(\alpha_X)F(\alpha_Y))}
 :"r" _{R(\beta_{F(X)}\beta_{F(Y)})}
 ,
 :@/_20pt/[d(2.5)l(3)] *+{RF(RF(X)RF(Y))}="m" _(.7){RF(\alpha_X\alpha_Y)}
 :"d" _{R(F_{RF(X),RF(Y)})}
 )
,
:@/_30pt/[d(2.5)r(6)] *+{RF(X)RF(Y)} _{\alpha_X\alpha_Y}
 (
 :@/_25pt/"r" _{R_{F(X),F(Y)}}
 ,
 :@/^20pt/"m" ^(.7){\alpha_{RF(X)RF(Y)}}
 )
)
}$$
Monoidality of $\beta$ follows from the commutative diagram:
$$\xygraph{
*+{F(R(X)R(Y))}="l" 
(
:@/^20pt/[u(2.5)r(2.5)] *+{FR(XY)}="u" ^{F(R_{X,Y})}
:@/^30pt/[d(2.5)r(6)] *+{XY}="r" ^{\beta_{XY}}
,
:@/_20pt/[d(2.5)r(2.5)] *+{FR(X)FR(Y)}="d" _{F_{R(X),R(Y)}}
:@/_30pt/"r" _{\beta_X\beta_Y}
,
:[u(1)r(2.5)] *+{FRF(R(X)R(Y))} ^{F(\alpha_{R(X)R(Y)})}
 (
 :[d(1)r(3)] *+{FR(FR(X)FR(Y))} ^{FR(F_{R(X),R(Y)})}
  (
  :@/_20pt/"u" _(.3){FR(\beta_X\beta_Y)}
  ,
  :@/^20pt/"d" ^(.3){\beta_{FR(X),FR(Y)}}
  )
 ,
 :[d(2)] *+{F(R(X)R(Y))} ^(.7){\beta_{F(R(X)R(Y))}}
 :@{=}"l"
 )
)
}$$
\end{proof}
The lax monoidal structure on $R$ allows us to transport algebras from $\cC$ to $\cZ(\cC)$. If $A$ is an algebra in $\cC$, $R(A)$ is an algebra in $\cZ(\cC)$ with the unit map 
$$\xymatrix{I\ar[r] & R(I)\ar[r]^{R(\iota_A)} & R(A)}$$ and the multiplication 
$$\xymatrix{R(A)\otimes R(A)\ar[r] & R(A\otimes A) \ar[r]^(.6){R(\mu_A)} & R(A)}$$
\begin{theo}\label{}
Suppose that the natural transformation $\beta$ of the adjunction is epi. Then for any algebra $A$ in a monoidal category $\cC$
$$Z(A) = C_l(R(A)).$$
\end{theo}
\begin{proof}
We are going to show that the adjunction 
$$\cC(F(Z),X)\simeq\cZ(\cC)(Z,R(X))$$
defines a monoidal equivalence between $\cZ(A)$ and $\cC_l(R(A))$. We start by constructing the functor $\cZ(A)\to\cC_l(R(A))$.
\newline
Let $Z\in\cZ(\cC)$ and $f:F(Z)\to A$ satisfy condition (\ref{cp}). Then the adjoint morphism $\tilde f:Z\to R(A)$, which is given by the composite 
$$\xymatrix{Z \ar[r]^{\alpha_Z} & RF(Z) \ar[r]^{R(f)} & R(A) ,}$$
satisfies condition (\ref{lcp}):
$$\xymatrix{ZR(A) \ar[dddd]_{z_{R(A)}} \ar@/^15pt/[rr]^{\tilde f1} \ar[r]_(.4){\alpha_Z1} & RF(Z)R(A) \ar[r]_{R(f)1} \ar[d] & R(A)^{\otimes 2} \ar[d] \ar@/^15pt/[ddr]^\mu & \\  
& R(F(Z)A) \ar[r]_{R(f1)} \ar[dd]_{R(z_A)} & R(A^{\otimes 2}) \ar[dr]_{R(\mu)} & \\
&&& R(A)\\
& R(AF(Z)) \ar[r]^{R(1f)} & R(A^{\otimes 2}) \ar[ur]^{R(\mu)} & \\
R(A)Z  \ar@/_15pt/[rr]_{1\tilde f} \ar[r]^(.4){1\alpha_Z} & R(A)RF(Z) \ar[r]^{1R(f)} \ar[u] & R(A)^{\otimes 2} \ar[u] \ar@/_15pt/[uur]_\mu &  \\
}$$
Here commutativity of the left rectangular face follows from the commutativity of
$$
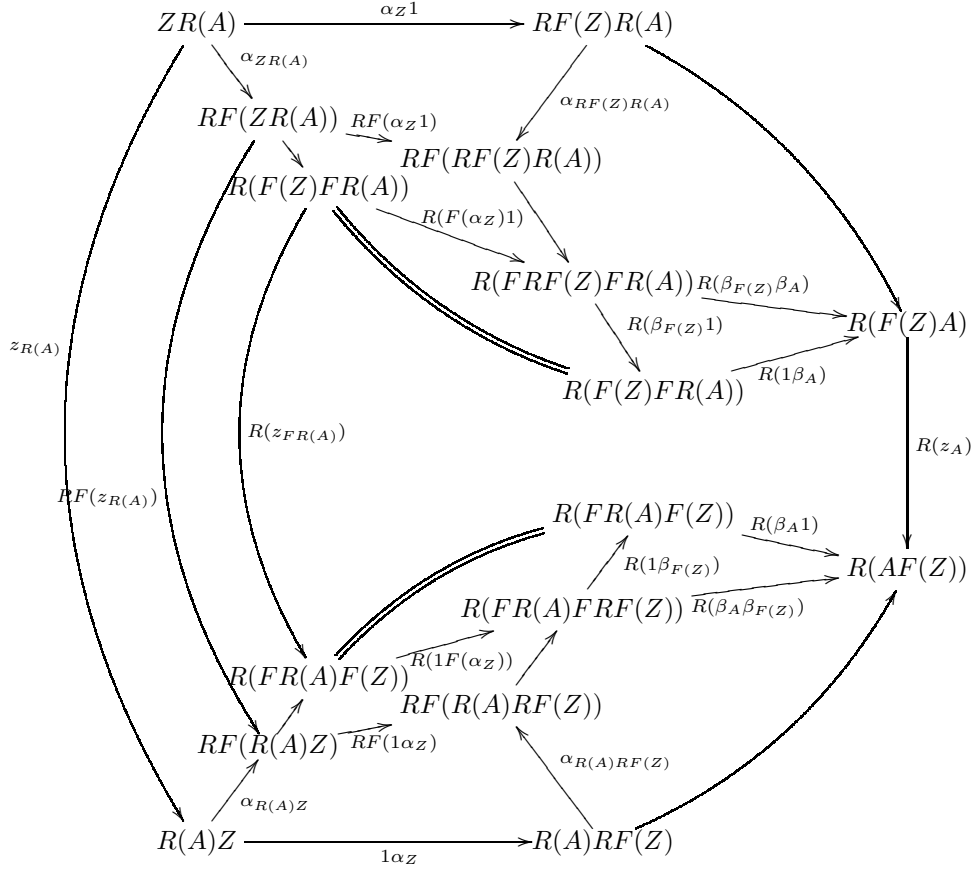
\xygraph{ !{0;/r3.2pc/:;/u4.3pc/::}[]*+{ZR(A)} 
(
:@/_50pt/[d(6)]*+{R(A)Z} _(.4){z_{R(A)}}
 (
 :[r(4)]*{R(A)RF(Z)} _{1\alpha_Z}
  (
  :@/_15pt/[u(2)r(3)]*+{R(AF(Z))}="rd"  
  ,
  :[lu]*+{RF(R(A)RF(Z))}="d" _{\alpha_{R(A)RF(Z)}}
  :[u(.7)r(.7)]*+{R(FR(A)FRF(Z))}="d2"
   (
   :"rd" _{R(\beta_A\beta_{F(Z)})}
   ,
   :[r(.7)u(.7)]*+{R(FR(A)F(Z))}="mrd" _(.6){R(1\beta_{F(Z)})}
   :"rd" ^{R(\beta_A1)}
   )
  )
 ,
 :[r(.7)u(.7)]*{RF(R(A)Z)}="ld" _{\alpha_{R(A)Z}}
  (
  :"d" _{RF(1\alpha_Z)}
  ,
  :[ur(.5)]*+{R(FR(A)F(Z))}="lm"
   (
   :"d2" _{R(1F(\alpha_Z))}
   ,
   :@/^15pt/@{=}"mrd"
   )
  )
 )
,
:[d(.7)r(.7)]*+{RF(ZR(A))} ^{\alpha_{ZR(A)}}
 (
 :@/_40pt/"ld" _(.6){RF(z_{R(A)})}
 ,
 :[dr(.5)]*+{R(F(Z)FR(A))} 
  (
  :@/_30pt/"lm" ^{R(z_{FR(A)})}
  ,
  :@/_15pt/@{=}[d(1.5)r(3.3)]*+{R(F(Z)FR(A))}="mu"
  :[r(2.5)u(.5)]*{R(F(Z)A)}="ru" _{R(1\beta_A)}
  :"rd" ^{R(z_A)}
  ,
  :[r(2.6)d(.7)]*+{R(FRF(Z)FR(A))}="um" ^{R(F(\alpha_Z)1)}
   (
   :"ru" ^{R(\beta_{F(Z)}\beta_A)}
   ,
   :"mu" ^{R(\beta_{F(Z)}1)}
   )
  )
 ,
 :[r(2.3)d(.3)]*+{RF(RF(Z)R(A))}="uum" ^{RF(\alpha_Z1)}
 :"um" 
 )
,
:[r(4)]*+{RF(Z)R(A)} ^{\alpha_Z1}
 (
 :"uum" ^{\alpha_{RF(Z)R(A)}}
 ,
 :@/^15pt/"ru" 
 )
) 
}$$
which, in its turn, follows from commutativity of 
$$\xymatrix{F(ZU) \ar[r] \ar[d]_{F(z_U)} & F(Z)F(U) \ar[d]^{z_{F(U)}} \\ F(UZ) \ar[r] & F(U)F(Z) }$$
(second left curved square in the diagram above). 
\newline
Thus we have a functor
$$\cZ(A)\to\cC_l(R(A)),\quad (Z,\zeta)\mapsto (Z,\overline\zeta).$$ Its monoidal property follows from monoidality of $F$ (and $R$).

We conclude by constructing a quasi-inverse functor $\cC_l(R(A))\to\cZ(A)$. For an object of $\cC_l(R(A))$, which is an object $Z\in\cZ(\cC)$ and a morphism $g:Z\to R(A)$, define $\tilde g:F(Z)\to A$ as the composition
$$\xymatrix{F(Z) \ar[r]^{F(g)} & FR(A) \ar[r]^{\beta_A} & A.}$$ The following diagram (together with the fact that $\beta$ is epi) shows that for $g$, satisfying condition (\ref{lcp}), $\tilde g$ satisfies condition (\ref{cp}):

$$
\xygraph{ !{0;/r3.5pc/:;/u4pc/::}[]*+{F(Z)FR(A)} 
(
:[ul]*+{F(Z)A}="tl" _{1\beta_A}
:@/_35pt/[d(6)]*+{AF(Z)}="bl" _{z_A}
 (
 :[rr]*{AFR(A)}="bm" ^{1F(g)}
 :[rr]*{A^2}="br" ^{1\beta_A}
 :@/_20pt/[u(3)r(2.5)]*{A}="r" _\mu
 ,
 :@/_15pt/"br"_{1\tilde g}
 )
,
:@/_20pt/[d(4)]*_{FR(A)F(Z)}="dl" _{z_{FR(A)}}
 (
 :"bl" _{\beta_A1}
 ,
 :[rr]*{FR(A)^2} ^{1F(g)}
  (
  :"bm" _{\beta_A1}
  ,
  :"br" ^{\beta_A^2}
  ,
  :[ur]*{F(R(A)^2)}="mrd"
  :[ur]*{FR(A)}="m" _{F(\mu)}
  :"r" ^{\beta_A}
  )
 ,
 :[ru]*{F(R(A)Z)}="mld"
 :"mrd" ^{F(1g)}
 )
,
:[dr]*+{F(ZR(A))}
 (
 :@/_10pt/"mld" _{F(z_{R(A)})}
 ,
 :[rr]*{F(R(A)^2)}="mru" ^{F(g1)}
 :"m" ^{F(\mu)}
 )
,
:[rr]*+{FR(A)^2} ^{F(g)1}
 (
 :"mru"
 ,
 :[ur]*{A^2}="tr" _{\beta_A^2}
 :@/^20pt/"r" ^\mu
 ,
 :[ul]*{FR(A)A}="tm" _{1\beta_A}
  (
  :"tr" ^{\beta_A1}
  :@{}"tl"
   (
   :"tm" ^(.6){F(g)1}
   ,
   :@/^15pt/"tr" ^{\tilde g1}
   )
 )
)
}$$
Here commutativity of the rightmost cells of the diagram is equivalent to the fact that $\beta_A:FR(A)\to A$ is a homomorphism of algebras, which follows from lemma \ref{monadj}. 
\end{proof}

\section{Morita invariance}\label{mi}

A right {\em module} over an algebra $A$ is a pair $(M,\nu)$, where $M$ is an object of $\cC$ and $\nu:M\otimes A\to M$
is a morphism ({\em action map}), such that $$\nu(\nu\otimes A) = \nu(M\otimes\mu).$$ A {\em homomorphism} of right
$A$-modules $M\to N$ is a morphism $f:M\to N$ in $\cC$ such that $$\nu_N(f\otimes A) = f\nu_M.$$

Right modules over an algebra $A\in\cC$ together with module homomorphisms form a category $\cC_A$. The forgetful
functor $\cC_A\to\cC$ has a right adjoint, which sends an object $X\in\cC$ into the {\em free} $A$-module $X\otimes A$,
with $A$-module structure defined by
$$\xymatrix{X\otimes A\otimes A \ar[r]^{I\mu} & X\otimes A.}$$
Since the action map $M\otimes A\to M$ is an epimorphism of right $A$-modules any right $A$-module is a quotient of a
free module.

Categories of modules over algebras are examples of module categories. A (left) {\em module} category \cite{qu} over a monoidal category $\cC$ is a category $\cM$ together with a functor (an {\em action functor}):
$$\cC\times\cM\to\cM,\quad (X,M)\mapsto X*M,$$
and a functorial isomorphism 
$$a_{X,Y,M}:X*(Y*M)\to (X\otimes Y)*M,\quad X,Y\in\cC,\ M\in\cM,$$
such that the diagram 
$$
\xygraph{ !{0;/r4.5pc/:;/u4.5pc/::}[]*+{X*(Y*(Z*M))} (
  :[u(.7)r(1.5)]*+{(X\otimes Y)*(Z*M)} ^{a_{X,Y,Z*M}}
  :[d(.7)r(1.5)]*+{((X\otimes Y)\otimes Z)*M}="r" ^{a_{XY,Z,M}}
  ,
  :[r(.5)d(.8)]*+!R(.3){X*((Y\otimes Z)*M)} _{X*a_{Y,Z,M}}
  :[r(2)]*+!L(.3){(X\otimes(Y\otimes Z))*M} _{a_{X,YZ,M}}
  : "r" _{\alpha_{X,Y,Z}*M}
)
}$$
commutes for any $X,Y,Z\in\cC,\ M\in\cM$. Here, for aesthetic reason,  we insert the associator $\alpha$ for the tensor product in $\cC$. 
Equivalently $\cM$ is a module category over $\cC$ if there is given a monoidal functor $\cC\to\End(\cM)$ to the monoidal category $\End(\cM)$ of endofunctors of $\cM$ (with monoidal structure given by composition of functors). 
\newline
A functor $F:\cM\to\cN$ between $\cC$-module categories is a {\em $\cC$-module} functor if it comes equipped with a natural collection of isomorphisms $F_{X,M}:F(X*M)\to X*F(M)$ such that the following diagram commutes:
$$
\xygraph{ !{0;/r4.5pc/:;/u4.5pc/::}[]*+{F(X*(Y*M))} (
  :[u(.7)r(1.5)]*+{F((X\otimes Y)*M)} ^{F(a_{X,Y,M})}
  :[d(.7)r(1.5)]*+{(X\otimes Y)*F(M)}="r" ^{F_{XY,M}}
  ,
  :[r(.5)d(.8)]*+!R(.3){X*F(Y*M)} _{F_{X,Y*M}}
  :[r(2)]*+!L(.3){X*(Y*F(M))} _{X*F_{Y,M}}
  : "r" _{a_{X,Y,M}}
)
}$$ Clearly, $\cC$-module structures on functors are composable: the composite of two $\cC$-module functors has a canonical structure of $\cC$-module functor. 
\newline
Note that if $(M,\nu)$ is a right module over an algebra $A$ in a monoidal category $\cC$ then for any $X\in\cC$ the tensor product $X\otimes M$ has a structure of a $A$-module $X\otimes:\nu:X\otimes M\otimes A \to X\otimes M$. Thus the category of (right) modules $\cC_A$ over an algebra $A$ in a monoidal category $\cC$ is a left $\cC$-module category with respect to the action functor
$$\cC\times\cC_A\to\cC_A,\quad (X,(M,\nu))\mapsto (X\otimes M, X\otimes\nu).$$
The forgetful functor $\cC_A\to\cC$ and its right adjoint have natural $\cC$-module structures, giving an adjoint pair of $\cC$-module functors. 

Let $\cM$ be a $\cC$-module category. With an object $(Z,z)$ of the monoidal centre of $\cC$ one can associate a functor
$$\cM\to\cM,\quad M\mapsto Z*M,$$ which comes equipped with a $\cC$-module structure
$$\xymatrix{Z*(X*M) \ar[r]^{a_{Z,X,M}} & (Z\otimes X)*M \ar[r]^{z_X*M} & (X\otimes Z)*M \ar[r]^{a_{X,Z,M}^{-1}} & X*(Z*M)}$$
Denote by $\End_\cC(\cM)$ the monoidal category of $\cC$-module endofunctors of $\cC$-module category $\cM$. The above construction defines a monoidal functor $E:\cZ(\cC)\to\End_\cC(\cM)$. 

Two algebras $A,B$ in a monoidal category $\cC$ are said to be {\em Morita equivalent} if their categories of right modules are equivalent as module categories over $\cC$. Here we are going to show that the full centre is an invariant of Morita equivalence. We will do it by extending the notion of full centre from algebras to module categories.

The {\em centre} of a (left) module category $\cM$ over a monoidal category $\cC$ is the object $Z(\cM)$ terminal in the comma category $E$$\downarrow$$I$, corresponding to the monoidal functor $E:\cZ(\cC)\to\End_\cC(\cM)$ and the (unit) algebra $I\in\End_\cC(\cM)$. In other words, the centre $Z(\cM)$ is the terminal object among pairs $(Z,f)$, where $Z\in\cZ(\cC)$ and $f_M:Z*M\to M$ is a collection of morphisms in $\cM$, natural in $M$, such that the following diagram commutes (for any $X\in\cC$ and $M\in\cM$):
\begin{equation}\label{cn}
\xymatrix{( Z\otimes X)*M \ar[r]^{a_{Z,X,M}^{-1}} \ar[dd]_{z_X*M} & Z*(X*M) \ar[rd]^{f_{X*M}} \\ & & X*M \\ (X\otimes Z)*M \ar[r]_{a_{X,Z,M}^{-1}} & X*(Z*M) \ar[ur]_{X*f_M} }
\end{equation}

\begin{prop}\label{cmc}
The centre $Z(\cM)$ of a module category $\cM$ over a monoidal category $\cC$ is a commutative algebra in $\cZ(\cC)$. 
\end{prop}
\begin{proof}
analogous to the proof of proposition \ref{algz}.
\end{proof}

\begin{theo}\label{}
Let $A$ be an algebra in a monoidal category $\cC$. Then
$$Z(\cC_A) = Z(A).$$
\end{theo}
\begin{proof}
it is enough to show that the comma category $E$$\downarrow$$I$ is monoidally equivalent to the category $\cZ(A)$ from remark \ref{subc}. We start by defining a functor $P:E$$\downarrow$$I\to\cZ(A)$. For $(Z,z)\in\cZ(\cC)$ with a natural collection $f_M:Z\otimes M\to M$ of morphisms of $A$-modules define a morphism $\overline f:Z\to A$ in $\cC$ to be the composite:
$$\xymatrix{Z \ar[r]^{1\iota} & Z\otimes A \ar[r]^{f_A} & A.}$$
Commutativity of the following diagram shows that the morphism $\overline f$ satisfies condition (\ref{cp}):
$$\xymatrix{
ZA \ar[rrr]^{\overline f1} \ar[dddrr]_{11\iota} \ar[dddd]_{z_A} \ar[drr]^{1\iota 1} &&& A^2 \ar[rdd]^\mu\\ 
&& ZA^2 \ar[ur]_{f_A1} \ar[rd]_{1\mu} \\
&&& ZA \ar[r]^{f_A} & A\\
& AZA \ar[drr]_{1f_A} & ZA^2 \ar[l]_{z_A1} \ar[ru]_{1\mu} \ar[rd]^{f_{A^2}}\\
AZ \ar[rrr]_{1\overline f} \ar[ru]^{11\iota} &&& A^2 \ar[ruu]_\mu
}$$
Here the left top square commutes by $A$-linearity of $f_*$, which says that the diagram
$$\xymatrix{ZMA \ar[r]^{f_M1} \ar[d]_{1\nu} & MA \ar[d]^\nu\\ ZM \ar[r]_{f_M} & M  }$$
commutes for all $M\in\cC_A$, while the triangle in the middle bottom of the diagram commutes by the $\cC$-module property of $f_*$, which is equivalent to the commutativity of the diagram
$$\xymatrix{ ZXM \ar[rr]^{f_{XM}} \ar[rd]_{z_X1} && XM\\ & XZM \ar[ru]_{1f_M} }$$ for all $M\in\cC_A$ and $X\in\cC$. We set $P(Z,f) = (Z,\overline f)$.

Now we construct the functor $Q:\cZ(A)\to E$$\downarrow$$I$. For $(Z,z)\in\cZ(\cC)$ with a morphism $g:Z\to A$ in $\cC$ define $\tilde g_M$ as the composite:
$$\xymatrix{ZM \ar[r]^{z_M} & MZ \ar[r]^{1g} & MA \ar[r]^\nu & M  }$$
The condition (\ref{cp}) for $g$ implies that $\tilde g_M$ is a morphism of $A$-modules:
$$\xymatrix{ ZMA \ar[r]^{z_M1} \ar[rd]_{z_{MA}} \ar[dd]_{1\nu} & MZA \ar[r]^{1g1} & MA^2 \ar[rr]^{\nu 1} \ar[dr]^{1\mu} && MA \ar[dd]^\nu\\
& MAZ \ar[u]_{1z_A} \ar[r]^{11g} \ar[d]_{\nu 1} & MA^2 \ar[r]^{1\mu} \ar[d]^{\nu 1} & MA \ar[dr]^\nu \\
ZM \ar[r]_{z_M} & MZ\ar[r]_{1g} & MA \ar[rr]_\nu && M
}$$
Obviously the collection $\tilde g_M$ is natural in $M\in\cC_A$. The $\cC$-module property of $\tilde g$ is almost self-evident:
$$\xymatrix{  
ZXM \ar[dr]^{z_{XM}} \ar[dd]_{z_X1} \\
& XMZ \ar[r]^{11g} & XMA \ar[r]^{1\nu} & XM\\
XZM \ar[ur]_{1z_M}
}$$ We set $Q(Z,g) = (Z,\tilde g)$. 

It is easy to see that the constructed functors are quasi-inverse to each other. Indeed, it follows from commutativity of the diagram:
$$\xymatrix{
& Z \ar[dl]_{1\iota} \ar[d]^{\iota 1} \ar[r]^f & A \ar[d]_{\iota 1} \ar[dr]^1\\
ZA \ar[r]_{z_A} & AZ \ar[r]_{1f} & A^2 \ar[r]_\mu & A
}$$
that $\overline{\tilde f} = f$, i.e. the composition $PQ$ is the identity. 

Similarly the diagram
$$
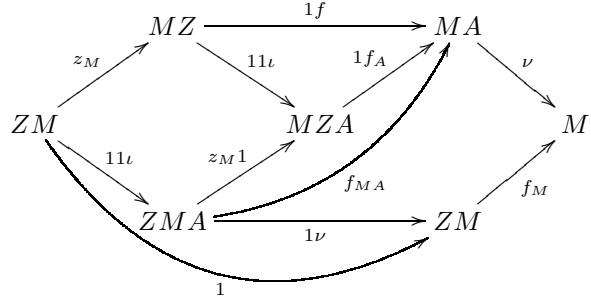
\xymatrix{
& MZ \ar[rr]^{1f} \ar[dr]^{11\iota} && MA \ar[dr]^\nu\\
ZM \ar[ur]^{z_M} \ar[rd]^{11\iota} \ar@/_40pt/[drrr]_1 && MZA \ar[ur]^{1f_A} && M\\
& ZMA \ar@/_20pt/[uurr]_{f_{MA}} \ar[ur]^{z_M1} \ar[rr]_{1\nu} && ZM \ar[ur]_{f_M}
}$$
implies that $\tilde{\overline f} = f$, i.e. the composition $QP$ is the identity.
\end{proof}

We immediately have the following.
\begin{corr}\label{}
The full centre of an algebra is Morita invariant.
\end{corr}

\section{Braided case}
Recall (from \cite{js}) that for a braided monoidal $\cC$ the tensor product functor 
$$T:\cC\times\cC\to\cC,\quad (X,Y)\mapsto X\otimes Y$$ is monoidal, with monoidal structure 
\begin{equation}\label{monstr}
\xymatrix{
T((X,Y)\otimes(Z,W))\ar@{=}[d] \ar[rr]^{T_{(X,Y),(Z,W)}} && T(X,Y)\otimes T(Z,W) \ar@{=}[d] \\
XZYW \ar[rr]^{1c_{Z,Y}1} && XYZW
}
\end{equation}
It turns out that this functor can be lifted to a monoidal functor $\cC\times\cC\to\cZ(\cC)$.
To construct this functor note that the braiding $c_{X,Y}:X\otimes Y\to Y\otimes X$ allows us to define braided monoidal functors
$$\iota_+:\cC\to\cZ(\cC),\quad X\mapsto (X,c_{X,-}),$$
$$\iota_-:\cC\to\cZ(\cC),\quad X\mapsto (X,c_{-,X}^{-1}).$$
These functors split the forgetful functor $F$:
$$F\iota_\pm\simeq Id_\cC,$$ with the natural isomorphism being the identity.
\newline
We can combine the functors $\iota_\pm$ into one
$$\cC\times\cC\to\cZ(\cC),\quad (X,Y)\mapsto\iota_+(X)\otimes\iota_-(Y).$$
The following lemma (essentially contained in \cite{mu}) says that this functor is monoidal, with the monoidal structure (\ref{monstr}).
\begin{lem}
The following diagram of monoidal functors commutes
$$\xymatrix{\cC\times\cC \ar[dr]_T \ar[rr] && \cZ(\cC) \ar[dl]^F \\ & \cC}$$
\end{lem}
\begin{proof}
We need to show that $1c_{Z,Y}1$ gives rise to a morphism $$\iota_+(XZ)\otimes\iota_-(YW)\to \iota_+(X)\otimes\iota_-(Y)\otimes\iota_+(Z)\otimes\iota_-(W)$$ in $\cZ(\cC)$. Monoidality of $\iota_\pm$ reduces it to the statement that $c_{Z,Y}$ is a morphism $\iota_+(Z)\otimes\iota_-(W)\to\iota_-(W)\otimes\iota_+(Z)$ in $\cZ(\cC)$, which follows from commutative diagram (here $U$ is an object of $\cC$):
$$\xymatrix{
ZWU \ar[rr]^{(c_{Z,-}|c_{-,W}^{-1})_U} \ar[ddd]_{c_{Z,W}1} \ar[rdd]_{c_{Z,WU}} && UZW \ar[ddd]^{1c_{Z,W}} \\
& ZUW \ar[lu]_{1c_{U,W}} \ar[ur]_{c_{Z,U}1} \ar[rdd]^{c_{Z,UW}}\\
& WUZ \\
WZU \ar[ru]^{1c_{Z,U}} \ar[rr]_{(c_{-,W}^{-1}|c_{Z,-})_U} && UWZ \ar[lu]^{c_{U,W}1}
}$$
\end{proof}

Let $\cM$ be a module category over a braided monoidal $\cC$. Following \cite{os} define two functors ({\em $\alpha$-inductions}) 
$\alpha_\pm:\cC\to\End_\cC(\cM)$ by $\alpha_\pm(X)(M) = X*M$ with $\cC$-module structures:
$$
\xymatrix{ \alpha_+(X)(Y*M) \ar[rr]^{\alpha_+(X)_{Y,M}} \ar@{=}[d] && Y*\alpha_+(X)(M) \ar@{=}[d] \\  X*(Y*M) \ar[d]_{a_{X,Y,M}} && Y*(X*M) \ar[d]^{a_{Y,X,M}} \\ (X\otimes Y)*M \ar[rr]_{c_{X,Y}*1} && (Y\otimes X)*M}$$
$$\xymatrix{ \alpha_-(X)(Y*M) \ar[rr]^{\alpha_-(X)_{Y,M}} \ar@{=}[d] && Y*\alpha_-(X)(M) \ar@{=}[d] \\  X*(Y*M) \ar[d]_{a_{X,Y,M}} && Y*(X*M) \ar[d]^{a_{Y,X,M}} \\ (X\otimes Y)*M \ar[rr]_{c^{-1}_{Y,X}*1} && (Y\otimes X)*M }$$
\begin{lem}\label{ialp}
The following diagram of monoidal functors commutes
$$\xymatrix{\cC\times\cC \ar[dr]_{\alpha_+\times\alpha_-} \ar[rr]^{\iota_+\times\iota_-} && \cZ(\cC) \ar[dl]^E \\ & \End_\cC(\cM)}$$
\end{lem}
\begin{proof}
this follows from the fact that the composite $E\circ\iota_\pm$ coincides with $\alpha_\pm$.
\end{proof}

\begin{prop}
Let $\cM$ be a module category over a braided monoidal category $\cC$. Then we have isomorphisms of Hom-spaces:
$$\cZ(\cC)(\iota_+(X)\otimes\iota_-(Y),Z(\cM))\simeq\End_\cC(\cM)(\alpha_+(X)\circ\alpha_-(Y),I).$$
\end{prop}
\begin{proof}
By the universal property of $Z(\cM)$, $\cZ(\cC)(\iota_+(X)\otimes\iota_-(Y),Z(\cM))$ coincides with $\End_\cC(\cM)(E(\iota_+(X)\otimes\iota_-(Y)),I)$, which coincides with $\End_\cC(\cM)(\alpha_+(X)\circ\alpha_-(Y),I)$ by lemma \ref{ialp}.
\end{proof}
This formula first appeared in the context of modular categories (see \cite{ffrs}), where it allows effective computation of the full centre.

\section{Modular case}\label{moc}

From now on we fix a ground field $k$. In this section all categories will be $k$-linear and finite (all hom-sets are finite dimensional vector spaces over $k$, composition is $k$-bilinear). The tensor product functor is bi-linear (linear in each argument). Slightly abusing the term we will call such categories {\em tensor}. All functors will be assumed $k$-linear (effect on morphisms is linear over $k$). 

An object $X^\sve$ is (left) {\em dual} to $X\in\cC$ if there exist morphisms $coev:I\to X\otimes X^\sve,\
ev:X^\sve\otimes X\to I$ such that the compositions
\begin{equation}\label{dua1}
\xymatrix{ X \ar[r]^(.3){coev X} & X\otimes X^\sve\otimes X \ar[r]^(.7){Xev} & X}
\end{equation}
\begin{equation}\label{dua2}
\xymatrix{ X^\sve \ar[r]^(.3){Xcoev} & X^\sve\otimes X\otimes X^\sve \ar[r]^(.7){evX} & X^\sve}
\end{equation}
are equal to the identity morphisms. A monoidal category is (left) {\em rigid} if all its objects have (left) duals. 

A rigid braided monoidal category $\cC$ is {\em ribbon} (or {\em tortile} \cite{sh,tu}) if it is equipped with a natural collection of isomorphisms
$\theta_X:X\to X$, satisfying the coherence axiom, which says that the diagram
\begin{equation}\label{ribaxiom}
\xymatrix{X\otimes Y \ar[r]^{\theta_{X\otimes Y}} \ar[d]_{c_{X,Y}} & X\otimes Y  \\ Y\otimes X
\ar[r]^{\theta_Y\otimes\theta_X} & X\otimes Y \ar[u]_{c_{Y,X}} }
\end{equation}
commutes for all $X,Y\in\cC$, and such that $\theta_{X^\sve} = \theta_X^\sve$ (the {\em self-duality axiom}).

Define the {\em trace} $tr(f)$ of an endomorphism $f:X\to X$ in a ribbon category $\cC$ as the
composition
$$\xymatrix{1 \ar[r]^(.4){coev_X} & X\otimes X^\sve \ar[r]^{\theta_X f\otimes X^\sve} & X\otimes X^\sve \ar[r]^{c_{X,X^\sve}} &
X^\sve\otimes X \ar[r]^(.7){ev_X} & 1}$$ The trace has the following properties (see \cite{tu} for
the proof):
$$tr(fg) = tr(gf), \quad f:X\to Y,\ g:Y\to X,$$ $$tr(f\otimes g) = tr(f)tr(g),\quad f:X\to X,\ g:Y\to Y,$$
$$tr(\lambda) = \lambda,\quad \lambda:1\to 1,$$ See \cite{js,tu} for details and proofs. 
\newline
There is a weaker notion (which does not require the presence of braiding) of so-called {\em spherical} monoidal category, where traces exist and have the right properties (see \cite{bw}). 

Recall that the {\em Deligne tensor product} $\cC\boxtimes\caD$ of two abelian $k$-linear categories is the abelian envelope of the tensor product of $\cC$ and $\caD$ as $k$-linear categories (or $\Vect$-enriched categories), i.e. category with objects being pairs $(X,Y),\ X\in\cC, Y\in\caD$ and hom spaces (from $(X,Y)$ to $(Z,W)$)) $\cC(X,Z)\otimes\caD(Y,W)$.

We call a braided monoidal category $\cC$ {\em non-degenerate} if the functor $\cC\boxtimes\cC\to\cZ(\cC)$ is an equivalence. 

Following \cite{tu} we call a ribbon (spherical) category $\cC$ {\em pure} if the bilinear pairing
$$\cC(X,Y)\otimes\cC(Y,X)\to k,\quad f\otimes g\mapsto tr(fg)$$ is non-degenerate for any $X,Y\in\cC$. Denote by $coev\in \cC(X,Y)\otimes\cC(Y,X)$
the canonical element of this pairing, which exists due to finite dimensiality of $\cC(X,Y)$. 

We will need the notion of ($k$-linear or $\Vect$-enriched) {\em coend} of a functor $S:\cC^{op}\times\cC\to\caD$ which we denote $\int^YS(Y,Y)$ (see \cite{mc,dk} for details). 
\begin{prop}\label{ra}
Let $\cC$ be a pure ribbon (spherical) category. Then the tensor product functor $T:\cC\boxtimes\cC\to\cC$ has the right adjoint:
$$R:\cC\to\cC\boxtimes\cC,\quad R(X) = \int^Y(X\otimes Y^\sve)\boxtimes Y =(X\boxtimes I)\otimes\tilde R,$$
where $\tilde R = \int^Y Y^\sve\boxtimes Y$.
\end{prop}
\begin{proof}
We need to define the adjunction natural transformations: $TR\to I,\ I\to RT$. 
\newline
The transformation $TR\to I$ is 
$$TR(X) \to \int^Y T((X\otimes Y^\sve)\boxtimes Y) = \int^Y X\otimes Y^\sve\otimes Y = X.$$
By the universal property of coend (see \cite{mc}) to define $$IX\boxtimes Y\to RT(X\boxtimes Y) = R(X\otimes Y) = \int^Z (X\otimes Y\otimes Z^\sve)\boxtimes Z$$ it is enough to present a dinatural collection of morphisms $(X\otimes Y\otimes Z^\sve)\boxtimes Z\to X\boxtimes Y$. By the definition of $\boxtimes$ $$\cC\boxtimes\cC((X\otimes Y\otimes Z^\sve)\boxtimes Z,X\boxtimes Y) = \cC(X\otimes Y\otimes Z^\sve,X)\otimes\cC(Z,Y).$$ The later coincides with $\cC(X\otimes Y,X\otimes Z)\otimes\cC(Z,Y)$. Now take the canonical element $coev\in\cC(Y,Z)\otimes\cC(Z,Y)$ and consider its image under the map $\cC(Y,Z)\otimes\cC(Z,Y)\to \cC(X\otimes Y,X\otimes Z)\otimes\cC(Z,Y)$, induced by tensoring with the identity morphism on $X$. Dinaturality of this collection is straightforward as well as the adjunction axioms.
\end{proof}
\begin{rem}
\end{rem}
If $\cC$ is semi-simple the coend $\int^Y Y^\sve\boxtimes Y$ always exists and coincides with $\oplus_{Y\in Irr(\cC)}Y^\sve\boxtimes Y$, where the sum runs over the representatives of the set $Irr(\cC)$ of isomorphism classes of simple objects of $\cC$. In particular, the transformation $\beta:TR\to I$ of the adjunction is epi. 
\medskip

Slightly changing the definition from \cite{tu} we call a semisimple monoidal category category {\em modular} if it is rigid, braided, ribbon
and non-degenerate.

Thus for an algebra $A$ in a modular category $\cC$ we have the following description of the full centre:
$$Z(A) = C_l(\iota_+(A)\otimes\tilde R),$$
which was used as the definition in \cite{ffrs,kr}.

\section{Examples}\label{gth}

Here we treat as examples the categories of vector spaces, graded by a group, and categories of representation of a group. 

Let $G$ be a group. Denote by $\cC(G)$ the category of $G$-graded vector spaces. This category is monoidal with respect to the tensor product of graded vector spaces: for $V=\oplus_{g\in G}V_g, U=\oplus_{g\in G}U_g$
$$V\otimes U = \oplus_{g\in G}(V\otimes U)_g,\quad (V\otimes U)_g = \oplus_{g_1g_2=g}V_{g_1}\otimes U_{g_2}.$$

An algebra in $\cC(G)$ is just a {\em $G$-graded algebra}, i.e. a
$G$-graded vector space $A = \oplus_{g\in G}A_g$ with multiplication, which preserves grading
$A_fA_g\subset A_{fg}$. 

We call a $G$-action on a vector space $V$ {\em compatible} with a $G$-grading $V = \oplus_{g\in G}V_g$ if $f(V_g) = V_{fgf^{-1}}$.
The following result is well-know (see for example \cite{da} for the proof).
\begin{prop}\label{mcgr}
The monoidal centre $\cZ(\cC(G))$ is isomorphic, as braided monoidal category, to the category $\cZ(G)$, whose objects
are $G$-graded vector spaces $X = \oplus_{g\in G}X_g$ together with a compatible $G$-action and with morphisms, which are graded
and action preserving homomorphisms of vector spaces. The tensor product in $\cZ(G)$ is the tensor product of
$G$-graded vector spaces with the $G$-action defined by
\begin{equation}\label{tp}
f(x\otimes y) = f(x)\otimes f(y),\quad x\in X, y\in Y.
\end{equation}
The monoidal unit is $I=I_e=k$ with trivial $G$-action.
\newline
The braiding is given by
\begin{equation}\label{br}
c_{X,Y}(x\otimes y) = f(y)\otimes x,\quad x\in X_f, y\in Y.
\end{equation}
\newline
For $Z\in\cZ(G)$ and $U\in\cC(G)$ the half-braiding $z_U:Z\otimes U\to U\otimes Z$ is given by
\begin{equation}\label{hb}
z_U(z\otimes u) = u\otimes g^{-1}(z),\quad u\in U_g. 
\end{equation}

\end{prop}

As an immediate application we have the following (see \cite{da} for details).
\begin{corr}
An algebra in the category $\cZ(G)$ is a $G$-graded associative algebra $C$ together with a $G$-action such that
\begin{equation}\label{ah}
f(ab) = f(a)f(b),\quad a,b\in C.
\end{equation}
An algebra $C$ in the category $\cZ(G)$ is commutative iff
\begin{equation}\label{co}
ab = f(b)a,\quad \forall a\in C_f, b\in C.
\end{equation}
\end{corr}

For a homogeneous element $v$ of a $G$-graded vector space $V$ the notation $|v|$ will denote its degree in $G$, i.e. $v\in V_{|v|}$. 
\begin{prop}\label{}
Let $A$ be an algebra in $\cC(G)$ (a graded $G$-algebra). The full centre of $A$ as an object of $\cZ(G)$ is the subspace of the space of functions $G\to A$ with homogeneous values:  
Let $A$ be a $G$-graded algebra. The full centre of $A$ as an object of $\cZ(G)$ is the subspace of the space of functions $G\to A$ with homogeneous values:  
$$Z(A) = \{ z:G\to A|\ \ az(g) = z(hg)a, \forall a\in A_h\}.$$
The $G$-grading on $Z(A)$ is given by $$Z(A)_f = \{ z\in Z(A)|\ |z(g)|=g|z(e)|g^{-1} = gfg^{-1}\}.$$ The $G$-action is $g(z)(f) = z(g^{-1}f)$. 
\\newline
The map $Z(A)\to A$ is the evaluation $z\mapsto z(e)$. 
\end{prop}
\begin{proof}
the condition (\ref{cp}) for the map $Z(A)\to A$ is equivalent to $z(e)a = ag^{-1}(z)(e)$ for any $z\in Z(A)$ and $a\in A_g$:
$$\xymatrix{
z\otimes a \ar@{|->}[r] \ar@{|->}[d] & z(e)\otimes a \ar@{|->}[r] & z(e)a  \\ a\otimes g^{-1}(z) \ar@{|->}[r] & a\otimes g^{-1}(z)(e) \ar@{|->}[r] & ag^{-1}(z)(e)
}$$ Thus condition (\ref{cp}) follows from the definition of $Z(A)$:
$ag^{-1}(z)(e) = az(g) = z(e)a$. 
\newline
Let $Z$ be an object of $\cZ(G)$ and $\zeta:Z\to A$ be a homorphism of $G$-graded vector spaces. Condition (\ref{cp}) implies that $\zeta(z)a = a\zeta(g^{-1}(z))$ for any $z\in Z_g$ and $a\in A$:
$$\xymatrix{
z\otimes a \ar@{|->}[r] \ar@{|->}[d] & \zeta(z)\otimes a \ar@{|->}[r] & \zeta(z)a  \\ a\otimes g^{-1}(z) \ar@{|->}[r] & a\otimes \zeta(g^{-1}(z)) \ar@{|->}[r] & a\zeta(g^{-1}(z))
}$$ 
Now we can define a morphism $Z\to Z(A)$ in $\cZ(G)$ by $z\mapsto\overline z$, where $\overline z:G\to A$ is given by $\overline z(g) = \zeta(g^{-1}(z))$. 
\end{proof}
By proposition \ref{algz}, $Z(A)$ is a (commutative) algebra in $\cZ(\cC)$. Indeed, the multiplication in $Z(A)$ is the componentwise product of functions $(zw)(g) = z(g)w(g)$. 

\medskip
We will be interested in a special class of algebras. An algebra $(A,\mu,\iota)$ in a rigid braided monoidal category $\cC$ is called {\em separable} if the following composition (denoted $e:A\otimes A\to 1$) is a non-degenerate pairing:
$$\xymatrix{A\otimes A \ar[r]^(.6)\mu & A \ar[r]^{\epsilon} & 1.}$$ Here $\epsilon$ is the composition
$$\xymatrix{A \ar[r]^(.3){I coev_A} & A\otimes A\otimes A^* \ar[r]^{\mu I} & A\otimes A^* \ar[r]^{c_{A,A^*}} & A^*\otimes A \ar[r]^{ev_A} & 1,}$$ where $coev_A$ and $ev_A$ are duality morphisms for $A$. 
Non-degeneracy of $e$ means that there is a morphism
$coev:1\to A\otimes A$ such that the composition
$$\xymatrix{A \ar[r]^{I coev} & A^{\otimes 3} \ar[r]^{eI} & A}$$ is the identity. It also implies that the similar composition
$$\xymatrix{A \ar[r]^{coev I} & A^{\otimes 3} \ar[r]^{Ie} & A}$$ is also the identity.

It is known that isomorphism (Morita) classes of indecomposable separable algebras in $\cC(G)$ correspond to pairs $(H,\gamma)$, where $H\subset G$ is a subgroup and $\gamma\in H^2(H,k^*)$ is a second cohomology class  with values in the group of invertible elements of the ground field. The class of $(H,\gamma)$ is represented by the {\em skew group algebra} $k[H,\gamma]$. As a vector space $k[H,\gamma]$ is spanned by $H$, with multiplication defined on the basis $e_g,\ g\in G$ by:
$$e_fe_g = \gamma(f,g)e_{fg}.$$ Here we use the symbol $\gamma$ for the cohomology class as well as its representing 2-cocycle. 

Indecomposable commutative separable algebras in $\cZ(G)$ were classified in \cite{da}. Here we briefly describe the result. Note that the identity component $A_e$ of an indecomposable commutative separable algebra $A\in\cZ(G)$ is a commutative separable algebra equipped with a $G$-action. It is well-known that such algebras are algebras of functions on transitive finite $G$-sets and are labelled by conjugacy classes of finite index subgroups of $G$ (stabilisers of $G$-sets). For a minimal idempotent $p$ in $A_e$ the algebra $pA$ is an indecomposable commutative separable algebra in $\cZ(St_G(p))$ with trivial component $pA_e=k$. As such it is a skew group algebra $k[F,\gamma]$, where $F\lhd H=St_G(p)$ is a normal subgroup. The $H$-action on $k[F,\gamma]$ is given by a function $\varepsilon:H\times F\to k^*$;
$$h(e_f) = \varepsilon_h(f)e_{hfh^{-1}}.$$
Thus we have the following (see \cite{da} for details).
\begin{theo}
Indecomposable commutative separable algebras in $\cZ(G)$ are of the form $A(H,F,\gamma,\varepsilon)$, where as a vector space, it is spanned by $a_{g,f}$, with $g\in G,f\in F$, modulo the relations
$$a_{gh,f} = \varepsilon_h(f)a_{g,hfh^{-1}},\quad \forall h\in H,$$ with the $G$-grading, given by $|a_{g,f}| =
gfg^{-1}$, the $G$-action $g'(a_{g,f}) = a_{g'g,f}$ and the multiplication
$$a_{g,f}a_{g',f'} = \delta_{g,g'}\gamma(f,f')a_{g,ff'}.$$
\end{theo}
Note that when $H=F$ the function $\varepsilon$ is completely defined by $\gamma$ (see \cite{da} for details). Thus we use the notation $A(H,H,\gamma)$ for $A(H,H,\gamma,\varepsilon)$. 

Now we calculate full centres of indecomposable separable algebras in $\cC(G)$. 
\begin{prop}\label{fcgr}
$Z(k[H,\gamma]) = A(H,H,\gamma)$.
\end{prop}
\begin{proof}
Let $z:G\to A$ be an element of $Z(k[H,\gamma])_f$. Since the values of $z$ are homogeneous and $|z(g)|=gfg^{-1}$, it should have the form $z(g) = \eta(g)e_{gfg^{-1}}$ for some function $\eta:G\to k$ with the support in $\{g\in G|\ gfg^{-1}\in H\}$. 
\newline
The condition $az(g) = z(hg)a, \forall a\in A_h$ is equivalent to the equation 
\begin{equation}\label{coef}
\gamma(h,gfg^{-1})\eta(g) = \eta(hg)\gamma(hgfg^{-1}h^{-1},h).
\end{equation}
Indeed, for $a=e_h$
$$az(g) = e_h\eta(g)e_{gfg^{-1}} = \gamma(h,gfg^{-1})\eta(g)e_{hgfg^{-1}}$$ should coincide with
$$z(hg)a = \eta(hg)e_{hgfg^{-1}h^{-1}}e_h = \eta(hg)\gamma(hgfg^{-1}h^{-1},h)e_{hgfg^{-1}}$$
In particular, for $|z|=f=e$, $z(g)=\eta(g)e_e$ with $\eta(hg)=\eta(g)$. So $Z(k[H,\gamma])_e$ coincides with the algebra $k(G/H)$ of functions on the $G$-set $G/H$. 
\newline
Let $p$ be the $\delta$-function of $H$. Then $pZ(k[H,\gamma])$ coincides with $k[H,\gamma]$. Indeed, the support of $z\in pZ(k[H,\gamma])$ is always in $H$ and solutions of the equation (\ref{coef}) do exist and are determined by $\eta(f)$. For $f\in H$ define $\eta_f:H\to k$ by $$\eta_f(h) = \gamma(h,f)\gamma^{-1}(hfh^{-1},h).$$ Then the map 
$$k[H,\gamma]\to pZ(k[H,\gamma]),\quad e_f\mapsto z_f,$$
where $z_f(h) = \eta_f(h)e_{hfh^{-1}}$, is an isomorphism. 
\end{proof}
For example, the unit algebra $I\in\cC(G)$ corresponds to the pair $(\{e\},1)$: $I=k[\{e\},1]$. Thus $Z(I)=A(\{e\},\{e\},1)$. 

Now we will deal with another series of examples.
Denote by $\Rep(G)$ the category of representations of $G$ over the ground field $k$. Note that $\Rep(G)$ is a symmetric tensor category over $k$. 
It is well-known that the monoidal centre $\cZ(\Rep(G))$ is equivalent (as a braided monoidal category) to $\cZ(G)$ (see for example \cite{os1}).

An algebra $A$ in $\Rep(G)$ (a {\em $G$-algebra}) is just an (associative, unital) algebra with an action of $G$ by algebra automorphisms. 
\begin{prop}\label{fcga}
The full centre $Z(A)\in\cZ(G)$ of an algebra $A\in\Rep(G)$ has the form $Z(A) = \oplus_{g\in G}Z_g(A)$, where
$$Z_g(A) = \{ x\in A|\ xa=g(a)x\ \forall a\in A\}$$
with the $G$ action, induced from $A$. 
\end{prop}
\begin{proof}
First we show that $Z(A)$ is an object of $\cZ(G)$. Indeed, $f(Z_g(A))=Z_{fgf^{-1}}(A)$:
$$f(x)a = f(xf^{-1}(a)) = f(gf^{-1}(a)x) = fgf^{-1}(a)f(x).$$
The morphism $Z(A)\to A$ is the direct sum of embeddings $Z_g(A)\subset A$. 
\newline
The diagram (\ref{cp}) commutes by the definition of $Z(A)$:
$$\xymatrix{
x\otimes a \ar@{|->}[r] \ar@{|->}[d] & x\otimes a \ar@{|->}[r] & xa \ar@{=}[d] \\ g(a)\otimes x \ar@{|->}[r] & g(a)\otimes x \ar@{|->}[r] & g(a)x 
}$$
Terminality of $Z(A)$ is also quite straightforward. For $\zeta:Z=\oplus_{g\in G}Z_g\to A$ and $z\in Z_g$ the condition (\ref{cp}) implies that $\zeta(z)\in Z_g(A)$:
$$\xymatrix{
z\otimes a \ar@{|->}[r] \ar@{|->}[d] & \zeta(z)\otimes a \ar@{|->}[r] & \zeta(z)a  \\ g(a)\otimes z \ar@{|->}[r] & g(a)\otimes\zeta(z) \ar@{|->}[r] & g(a)\zeta(z)
}$$
\end{proof}
Proposition \ref{algz} implies that $Z(A)$ is a (commutative) algebra in $\cZ(\cC)$, which can be checked directly. Indeed, $Z_f(A)Z_g(A)\subset Z_{fg}(A)$
$$xya = xg(a)y = fg(a)xy,\quad x\in Z_f(A), y\in Z_g(A), a\in A.$$

It is known (see for example \cite{os1}) that Morita classes of indecomposable algebras in $\Rep(G)$ are in 1-to-1 correspondence with Morita classes of indecomposable separable algebras in $\cC(G)$, i.e. they correspond to pairs  $(H,\gamma)$, where $H\subset G$ is a subgroup and $\gamma\in H^2(H,k^*)$.
 A representative for the class, corresponding to a pair $(H,\gamma)$, can be constructed as follows. Let $V$ be an irreducible projective representation of $H$ with the Schur multiplier $\gamma$, i.e. there is given a homomorphism of algebras $\rho:k[H,\gamma]\to End(V)$ such that the centraliser of the image of $\rho$ is trivial. Then $End(V)$ is an $H$-algebra. Define $A(H,\gamma)$ to be the $G$-algebra, induced from the $H$-algebra $End(V)$:
 $$A(H,\gamma) = ind^G_H(End(V)) = \{a:G\to End(V)|\ a(hg)=h(a(g)),\ \forall h\in H\},$$ with the $G$-action given by $f(a)(g) = a(gf)$  (see \cite{da1} for details). 
\begin{prop}\label{}
$Z(A(H,\gamma)) = A(H,H,\gamma)$.
\end{prop}
\begin{proof}
By the definition of the full centre of a $G$-algebra, $Z_e(A)$ is the ordinary centre, i.e. the centre of $A$ as an algebra in the category of vector spaces. Thus
$$Z_e(A(H,\gamma)) = ind^G_H(Z(End(V)) = ind^G_H(k) = k(G/H).$$
Let $p\in Z_e(A(H,\gamma))$ be the $\delta$-function of $H$. Note that 
$$pA(H,\gamma) = ind^H_H(End(V)) = End(V).$$
Thus $pZ(A(H,\gamma))$, which coincides with the full centre $Z(End(V))$ of the $H$-algebra $End(V)$, is isomorphic to $k[H,\gamma]$. Indeed, 
$\rho(e_h)$ belongs to $Z_h(End(V))$, so $\rho$ defines a homomorphism $k[H,\gamma]\to End(V)$ of algebras in $\cZ(H)$. Since $Z_e(End(V))=k$ and since $\rho(e_h)^{-1}Z_h(End(V)) = Z_e(End(V))$, this is an isomorphism.
\end{proof}
For example, the unit algebra $I\in\Rep(G)$ corresponds to the pair $(G,1)$. Thus $Z(I)=A(G,G,1)$. 

\begin{rem}
\end{rem}
It can be seen from propositions \ref{fcgr},\ref{fcga}, that the only indecomposable commutative algebras in $\cZ(G)$, that appear as full centres, are those of the form $A(H,H,\gamma)$. This is related to the fact that they have trivial categories of so-called local (or dyslectic) modules (see \cite{da} for details). An explanation of this will be given in a subsequent paper.

\end{document}